\documentclass{article}
\usepackage{a4wide}
\usepackage{amsmath,amsfonts,mathscinet,graphicx}
\usepackage{amsthm}
\usepackage{amssymb}
\usepackage{yfonts}
\usepackage{todonotes}
\newtheorem{theorem}{Theorem}[section]
\newtheorem{lemma}[theorem]{Lemma}
\newtheorem{proposition}[theorem]{Proposition}
\newtheorem{corollary}[theorem]{Corollary}

\newtheorem{remark}[theorem]{Remark}
\newcommand{\tn}{|\mspace{-1mu}|\mspace{-1mu}|}

\numberwithin{equation}{section}

\title{\bf Stabilized nonconforming finite element methods for data assimilation
  in incompressible flows}
\author{Erik Burman\footnote{Department of Mathematics, University College London, London, UK-WC1E  6BT, United Kingdom, e.burman@ucl.ac.uk}
\; and \,
Peter Hansbo
\footnote{Department of Mechanical Engineering, J\"onk\"oping University,
SE-55111 J\"onk\"oping, Sweden, peter.hansbo@ju.se}}

\date{}

\begin{document}
\maketitle
\begin{abstract}
We consider a stabilized nonconforming finite element method for data
assimilation in incompressible flow subject to the Stokes' equations. The method uses a primal dual
structure that allows for the inclusion of nonstandard data. Error
estimates are obtained that are optimal compared to the conditional
stability of the ill-posed data assimilation problem.
\end{abstract}

\section{Introduction}\label{sec:intro}
The design of computational methods for the numerical approximation of the Stokes' system of equations modelling
creeping incompressible flow is by and large well understood in the
case where the underlying problem is well-posed. Indeed, provided suitable
boundary conditions are set, the system of equations are known to
satisfy the hypotheses of the Lax-Milgram lemma and Brezzi's theorem
ensuring well-posedness of velocities and pressure. These theoretical
results then underpin much of the theory for the design of stable and
accurate finite element methods for the Stokes system \cite{GR86,BBF13}.

In many cases of interest in applications, however, the necessary data
for the theoretical results to hold are not known; this is the case
for instance in data assimilation in atmospheric sciences or
oceanography. Instead of knowing the solution on the boundary, data in
the form of measured values of velocities may be known in some other
set. It is then not obvious how best to apply the theory
developed for the well-posed case. A classical approach is to rewrite the system as an
optimisation problem and add some regularization, making the problem
well-posed on the continuous level and then approximate the
well-posed problem using known techniques. For examples of methods
using this framework see \cite{ASH09} and \cite{BD14}.

In this paper we advocate a different approach in the spirit of
\cite{Bu13,Bu14}.
The idea is to formulate the optimization problem on the continuous
level, but without any regularization. We then discretize the
ill-posed continuous problem and instead regularize the discrete
solution.
This leads to a method in the spirit of stabilized finite element
methods where the properties of the different stabilizing operators
are well studied. An important feature of this approach is that it
eliminates the need for a perturbation analysis on the continuous level
taking into account the Tikhonov regularization and perturbations in
data, that the discretization error then has to match. In our case we
are
only interested in the discretization error and the perturbations in
data. This allows us to derive error estimates that are optimal in the
case of unperturbed data in a similar fashion as for the well-posed case. 

We exemplify the theory in a model case for data assimilation
where data is given in some subset of the computational domain instead
of the boundary, and we obtain error estimates using a conditional
stability result in the form of a three ball inequality due to
Lin, Uhlmann, and Wang \cite{LUV10}. A particular feature of the method formulated for the
integration of data in the bulk (and not on the boundary), is that the
dual adjoint problem does not require any regularization on the
discrete level. Indeed, the adjoint equation is inf--sup stable, similarly
to the case of elliptic problems on non-divergence form discussed
in \cite{WW15}.

The rest of the paper can be outlined as follows. First, in Section \ref{sec:stokes}, we introduce the Stokes' problem that we are interested
in and propose the continuous minimization problem. Then, in Section \ref{fem}, we present the non-conforming finite element method and prove
some preliminary results. In Section \ref{theory} we prove the fundamental
stability and convergence results of the formulation. Finally we show
the performance of the approach on some numerical examples.

\section{Stokes equations}\label{sec:stokes}
Let $\Omega$ be a polygonal (polyhedral) domain in $\mathbb{R}^d$, $d=2$ or $3$.
We are interested in computing solutions to the Stokes' system
\begin{equation}\label{stokes}
\begin{array}{rcl}
-\Delta u + \nabla p &=& \mathfrak{f} \quad \mbox{ in } \Omega \\[3mm]
\nabla \cdot u & = & \mathfrak{g} \quad \mbox{ in } \Omega.
\end{array}
\end{equation}
Typically these equations are then equipped with
suitable boundary conditions and are known to be well-posed using the Lax-Milgram Lemma
for the velocities and Brezzi's theorem for the pressures. It is also
known that the following continuous dependence estimate holds, here 
given under the assumption of homogeneous Dirichlet conditions on the boundary.
\begin{equation}\label{Cont_dep_well_posed}
\|u\|_{H^1(\Omega)} + \|p\|_{\Omega} \lesssim
\|\mathfrak{f}\|_{H^{-1}(\Omega)} + \|\mathfrak{g}\|_{\Omega},
\end{equation}
where we used the notation
  $\|x\|_\Omega:=\|x\|_{L^2(\Omega)}$ and $a \lesssim b$ for $a \leq
  C b$ with $C>0$. 

Observe that for any solution to the equations \eqref{stokes} and in any closed ball $B_R \subset \Omega$ there holds
\begin{equation}\label{eq:elliptic_reg}
(u,p)\vert_{B_R} \in  [H^2(B_R)]^d \times
H^1(B_R).
\end{equation}
Provided $\mathfrak{f} \in [L^2(\Omega)]^d$ and $\mathfrak{g} \in
H^1(\Omega)$. See for instance \cite[Proposition 3.2]{Ser15}.

We will in the following make the stronger assumption that $(u,p) \in [H^2(\Omega)]^d \times
H^1(\Omega)$. Observe that this is not a strong
  assumption for the particular problem we will study below, since the domain $\Omega$ here is somewhat arbitrary and not
necessarily determined by a physical geometry. Indeed the only situation in which this assumption can fail is
when the boundary of $\Omega$ coincides with a physical boundary with
a corner.

Herein the main focus will be on methods that allow for the accurate
approximation of the solution under the much weaker stability
estimates that remain valid in the case of ill-posed
  problems where \eqref{Cont_dep_well_posed} fails.

A situation of particular
interest is the case where the boundary data $g_D$ is
known only on a portion $\Gamma_D$ of $\partial \Omega$ and nothing is
known of the boundary conditions on the remaining part $\Gamma_D'
:= \partial \Omega \setminus \Gamma_D$. This lack of boundary information
makes the problem ill-posed and we assume that some other data is
known such as:
\begin{itemize}
\item The normal stress in some part of the boundary $\Gamma_N
  \subset \partial \Omega$ and $\Gamma_N \cap \Gamma_D \ne \emptyset$,
\begin{equation}\label{cauchy}
(- n \cdot \nabla u + pn) \cdot n = \psi.
\end{equation}
We will refer to this problem as the \emph{ Cauchy problem } below.
\item The measured value of $(u,p)$ in some subdomain $\omega \subset
  \Omega$. We will refer to this problem as the \emph{data
    assimilation problem} below.
\end{itemize}
In the first case it is known that if a solution exists, then $g_D =
\psi = 0$ implies $u=0$, $p=0$ in $\Omega$ by unique continuation \cite{FL96},
however, no quantitiative estimates appear to exist in the literature
for the pure Cauchy problem;
see \cite{BEG13} for results using additional measurements on the
boundary. In the second case stability may be proven in the form of a
three balls inequality and associated local stability estimates, see
\cite{LUV10,BEG13}. For completeness of the analysis we focus on the second case
for the error estimates below. In particular we consider the case
where no data are known on the boundary, i.e. $\Gamma_D = \Gamma_N = \emptyset$.
In the data assimilation case the following Theorem from \cite{LUV10}
provides us with a conditional stability estimate. Assuming an optimal conditional stability estimate for the Cauchy problem in
the spirit of \cite{ARRV09}, it is straightforward to
extend the anaysis to this case
following \cite{Bu15}. 

\begin{theorem}\label{thm:3sphere}(Conditional stability for the
  Stokes' problem)
There exists a positive number $\tilde R<1$ such that if
$0<R_1<R_2<R_3\leq R_0$ and $R_1/R_3<R_2/R_3<\tilde R$, then if
$B_{R_0} (x_0) \subset \Omega$
\[
\int_{B _{R_2}(x_0)} |u|^2 ~\mbox{d}x \leq C \left( \int_{B _{R_1}(x_0)} |u|^2 ~\mbox{d}x\right)^{\tau} \left(\int_{B_{R_3}(x_0)} |u|^2 ~\mbox{d}x \right)^{1-\tau}
\]
for $(u,p) \in [H^1(B_{R_0}(x_0))]^{d+1}$, satisfying \eqref{stokes} with
$\mathfrak{f}=\mathfrak{g}=0$  in $B_{R_0}(x_0)$, where the constant $C$ depends on $R_2/R_3$
and $0<\tau<1$ depends on $R_1/R_3$, $R_2/R_3$ and $d$. For fixed
$R_2$ and $R_3$, the exponent $\tau$ behaves like $1/(-\log(R_1))$
when $R_1$ is sufficiently small.
\end{theorem}
\begin{proof}
For the proof we refer to  \cite{LUV10}.
\end{proof}
In the data
assimilation problem corresponding to Theorem \ref{thm:3sphere} measured data $u_M:\omega \mapsto \mathbb{R}^d$ are
available in $\omega$ such that $u_M$ satisfies \eqref{stokes} in
$\omega$ and there exists $u$ defined on $\Omega$ satisfying
\eqref{stokes} such that $u\vert_\omega = u_M$. Our objective is to
design a method for the reconstruction of $u$, given $\tilde u_M := u_M + \delta u$,
where $\delta u \in [L^2(\omega)]^d$ is a perturbation of the exact data resulting from measurement
error or interpolation of pointwise measurements inside
$\omega$. Observe that the considered configuration is also closely
related to a pure boundary control problem, where we look for data on
the boundary such that $u = u_M$ in the subset $\omega$.

We will first cast the problem \eqref{stokes}, with the notation $\mathfrak{f}=f$ and with $\mathfrak{g}=0$, on weak form. For the derivation of the weak formulation we introduce the spaces
$
V:= \{v \in [H^1(\Omega)]^d \}$ and
$W:= \{v \in [H^1_0(\Omega)]^d\}
$ for velocities and $Q:=L^2(\Omega)$ and $Q_0:=L^2_0(\Omega)$, where
the zero--subscript in the second case as usual indicates that the functions
have zero integral over $\Omega$.

We may the multiply the first equation of \eqref{stokes} by $w \in W$
and first integrate over $\Omega$ and then apply Green's formula to obtain
\[
\int_\Omega \nabla u:\nabla w~\mbox{d}x - \int_{\Omega} p \nabla \cdot w
~\mbox{d}x = \int_\Omega f w ~\mbox{d}x,\quad \forall w \in W
\]
similarly we may multiply the second equation by $q \in L^2(\Omega)$
and integrate over $\Omega$ to get
\[
\int_{\Omega} q \nabla \cdot u
~\mbox{d}x = 0.
\]
Introducing the forms
\[
a(u,w) := \int_\Omega \nabla u:\nabla w~\mbox{d}x ,
\]
\[
b(p,w) = -\int_{\Omega} p \,\nabla \cdot w
~\mbox{d}x 
\]
and 
\[
l(w) :=  \int_\Omega f w ~\mbox{d}x 
\]
we may formally write the problem as: find $(u,p) \in V \times Q_0$ such that $u
\vert_{\omega} = u_M$ and
\begin{align}
a(u,w) + b(p,w) &= l(w),\quad \forall w \in W \\
b(y,u) & = 0,\quad \forall y \in Q.
\end{align}
Observe that this problem is ill-posed. In particular observe that we are not allowed to
test with $w=u$ because of the homogeneous Dirichlet conditions set on
the functions in $W$. To regularize the problem we cast it on the form of a
minimization problem, first writing
\[
A[(u,p),(w,y)] := a(u,w) + b(p,w) - b(y,u)
\]
and then introducing the Lagrangian
\[
\mathcal{L}[(u,p),(z,x)] := \frac12 m(u - \tilde u_M,u - \tilde u_M) + A[(u,p),(z,x)]-l(z),
\]
where $m(\cdot,\cdot)$ is a bilinear form that depends on what data we
wish to integrate. For the data assimilation problem that is our main concern we
simply have
\[
m(u,v) := \gamma_M\int_{\omega} uv ~\mbox{d}x,
\]
where $\gamma_M>0$ is a free parameter. We will also use the notation
\[
(u,v)_\omega  := \int_{\omega} uv ~\mbox{d}x.
\]
The optimality system of the associated constrained minimization
problem takes the form
\begin{align}\label{eq:min1}
A[(u,p),(w,y)] & = l(w) \\
A[(v,q),(z,x)] + m(u,v) & = m(\tilde u_M,v). \label{eq:min2}
\end{align}
This problem is ill-posed in general, but in the data assimilation
case we know that if a solution exists and $l(w)=0$
then this solution must satisfy the conditional stability of Theorem
\ref{thm:3sphere}. A consequence of this is that if the system admits a solution
$(u,p) \in V \times L^2(\Omega)$ for the exact
data $u_M$, then this solution is unique. To show this assume that
there are two solutions $u_1\in V$ and $u_2\in V$ that solve
\eqref{eq:min1}--\eqref{eq:min2},
then $v = u_1-u_2 \in V$ solves the homogenous Stokes' equation and has
$v\vert_\omega = 0$ and the uniqueness is a consequence of unique
continuation based on Theorem \ref{thm:3sphere}. Below we will assume that there
exists a unique solution $(u,p) \in [H^2(\Omega)]^d \times
H^1(\Omega)$ that satisfies \eqref{stokes} in $\Omega$ with $u = u_M$
in $\omega$.

\section{The nonconforming stabilized method}\label{fem}
Let $\{\mathcal{T}_h\}_h$ denote a family of shape regular and quasi uniform
tesselations of $\Omega$ into nonoverlapping simplices, such that for
any two different simplices $\kappa$, $\kappa' \in \mathcal{T}_h$, $\kappa \cap
\kappa'$ consists of either the empty set, a common face or a common
vertex. The outward pointing normal of a simplex $\kappa$ will be
denoted $n_{\kappa}$. We denote the set of element faces in $\mathcal{T}_h$ by
$\mathcal{F}$ and let $\mathcal{F}_i$ denote the set of interior faces $F$ in
$\mathcal{F}$.
To each face
$F$ we associate a unit normal vector, $n_F$. For interior faces its orientation is arbitrary, but fixed. On the boundary
$\partial \Omega$ we identify $n_F$ with the outward pointing normal
of $\Omega$.
We define the jump over interior faces $F \in
\mathcal{F}_i$ by $[v]\vert_F:= \lim_{\epsilon \rightarrow
  0^+} (v(x\vert_F- \epsilon n_F) - v(x\vert_F+ \epsilon n_F))$
and for faces on the boundary, $F \in \partial \Omega$, we let
$[v]\vert_F := v \vert_F$. Similarly we define the average of a function over
an interior
face $F$ by $\{v \}\vert_F := \tfrac12 \lim_{\epsilon \rightarrow
  0^+} (v(x\vert_F- \epsilon n_F) + v(x\vert_F+ \epsilon n_F))$ and
for $F$ on the boundary we define $\{v \}\vert_F := v \vert_F$.
The classical nonconforming space of piecewise affine
finite element functions (see \cite{CR73}) then reads
$$
X_h := \{v_h \in L^2(\Omega): \int_{F} [v_h]~
\mbox{d}s = 0,\, \forall F \in \mathcal{F}_i
\mbox{ and } v_h\vert_{\kappa} \in \mathbb{P}_1(\kappa),\,
\forall \kappa \in \mathcal{T}_h \}
$$
where $\mathbb{P}_1(\kappa)$ denotes the set of polynomials of degree less than
or equal to one restricted to the element $\kappa$, 
and with homogeoneous Dirichlet boundary conditions
$$
X_h^{0} := \{v_h \in L^2(\Omega): \int_{F} [v_h]~
\mbox{d}s = 0,\, \forall F \in \mathcal{F}
\mbox{ and } v_h\vert_{\kappa} \in \mathbb{P}_1(\kappa),\,
\forall \kappa \in \mathcal{T}_h \}.
$$
We may then define the spaces $V_h := [X_h]^d$ and $W_h :=
[X_h^{0}]^d$. For the pressure spaces we define
\[
Q_h := \{q_h \in L^2(\Omega): q\vert_\kappa \in \mathbb{R}, \forall \kappa \in
\mathcal{T}_h \} \mbox{ and } Q_h^0 := Q_h \cap L^2_0(\Omega).
\]
To make the notation more compact we introduce the composite spaces
$\mathcal{V}_h:= V_h \times Q_h^0$ and $\mathcal{W}_h:= W_h \times Q_h$.

\subsection{Finite element formulation}
By writing the equations \eqref{eq:min1}--\eqref{eq:min2} with
arguments in the discrete spaces, the formulation may now naively be written:
 find $(u_h,p_h) \times (z_h,x_h) \in \mathcal{V}_h \times
 \mathcal{W}_h$ such that,
\begin{align}\label{eq:FEMmin1}
A_h[(u_h,p_h),(w_h,y_h)] & = l(w) \\
A_h[(v_h,q_h),(z_h,x_h)] +m(u_h,v_h) & = m(\tilde u_M,v_h). \label{eq:FEMmin2}
\end{align}
for all $(v_h,q_h) \times (w_h,y_h) \in \mathcal{V}_h \times
 \mathcal{W}_h$. The discrete bilinear form is defined by
\begin{equation}\label{Compactform}
A_h[(u_h,p_h),(w_h,y_h)] := a_h(u_h,w_h) +b_{h}(p_h,w_h) - b_h(y_h,u_h)
\end{equation}
where the forms are defined by
\[
a_h(u_h,w_h) = \sum_{\kappa \in \mathcal{T}_h} \int_\kappa \nabla u_h:
\nabla w_h ~\mbox{d}x,
\]
\[
b_{h}(p_h,w_h) = -\sum_{\kappa \in \mathcal{T}_h} \int_{\kappa} p_h \,\nabla \cdot w_h
~\mbox{d}x.
\]
To obtain a stable formulation we need to add stabilizing terms. This
can be done in several different ways, resulting in different methods
with different stability, accuracy and conservation properties. Our
choice herein has been guided by the principle that stabilization
is added only if it is necessary for accuracy and has minimal
influence
on the conservation properties of the scheme. We will also comment on
some variants.
For the primal velocities we suggest to use the standard jump stabilization
that has been shown to stabilize the Crouzeix-Raviart element in a
number of applications \cite{HL02,HL03,BH05},
\begin{equation}\label{vel_stab}
s_{j,t}(u_h,v_h)  := \sum_{F \in \mathcal{F}_i} \int_{F} 
h_F^{t} [u_h][v_h] ~\mbox{d}s.
\end{equation}
For the pressure on the other hand we propose to use the following
weak penalty term
\begin{equation}\label{pres_stab}
s_{p,t}(p_h,q_h) := \int_\Omega h^t p_h q_h ~\mbox{d}x.
\end{equation}
We also propose the compact form: find $(U_h,Z_h) \in \mathcal{V}_h
\times \mathcal{W}_h$, where $U_h := (u_h,p_h) \in V_h \times Q_h^0$
and  $Z_h := (z_h,x_h) \in W_h \times Q_h$, such that,
\begin{equation}\label{FEM:compact}
\mathcal{A}_h[(U_h,Z_h),(X_h,Y_h)] + \mathcal{S}_h[(U_h,Z_h),(X_h,Y_h)] + m(u_h,v_h) = l(w_h) + m(\tilde u, v_h)
\end{equation}
for all $(X_h,Y_h) \in \mathcal{V}_h
\times \mathcal{W}_h$, $X_h:=(v_h,q_h)$ and $Y_h:=(w_h,y_h)$. The bilinear forms are then
given by
\begin{equation}\label{Aglobal}
\mathcal{A}_h[(U_h,Z_h),(X_h,Y_h)] := A_h[(u_h,p_h),(w_h,y_h)]+A_h[(v_h,q_h),(z_h,x_h)]
\end{equation}
and 
\begin{equation}\label{Sglobal}
\mathcal{S}_h[(U_h,Z_h),(X_h,Y_h)]:= S_p[(u_h,p_h),(v_h,q_h)]-S_a[(z_h,x_h),(w_h,y_h)],
\end{equation}
where $S_a$ and $S_p$ are positive semi-definite, symmetric bilinear
forms. 
In the following, we shall also make use of the following bilinear form
\begin{align}
\mathcal{G}[(U_h,Z_h),(X_h,Y_h)]:= &{} \mathcal{A}_h[(U_h,Z_h),(X_h,Y_h)]  \\
& {} +\mathcal{S}_h[(U_h,Z_h),(X_h,Y_h)] + m(u_h,v_h) .\nonumber
\end{align}

The precise design of the regularization is
  problem dependent. For the Cauchy problem, the velocities must be
stabilized both for the forward and the adjoint problems. This is not
necessary in the data
assimilation case, where the stabilizing terms takes the form
\begin{equation}\label{eq:stab_p}
S_p[(u_h,p_h),(v_h,q_h)] := \gamma_u s_{j,-1}(u_h,v_h) + \gamma_p
s_{p,2}(p_h,q_h),\, \gamma_u >0,\, \gamma_p \ge 0
\end{equation}
and
\begin{equation}\label{eq:stab_a}
S_a[(z_h,x_h),(w_h,y_h)] := \gamma_x s_{p,0}(x_h,y_h),\, \gamma_x \ge 0.
\end{equation}
Observe that the minimal stabilization that allows for optimal error
estimates is $\gamma_u>0$, $\gamma_p = \gamma_x =0$. In the analysis
below we will focus on this case, noting that the case with added
pressure stabilization follows in a similar way, but is slightly more
elementary. From the theoretical point of view the choice $\gamma_p >
0$ has no detrimental effect, neither on conservation nor on the
accuracy of the primal solution. The choice $\gamma_x >0$ 
on the other hand perturbs both local and global conservation, but
still allows for optimal error estimates. The interest of the addition
of the pressure stabilization stems from the possibility of
eliminating the pressure and we briefly discuss the resulting
formulation before proceeding with the analysis.

\subsection{Elimination of the pressure}
Consider the dual mass conservation equation in the 
formulation \eqref{FEM:compact} with the stabilization given by
\eqref{eq:stab_p} and \eqref{eq:stab_a} and $\gamma_p>0$,
\[
b(q_h,z_h) + s_{p,2}(p_h,q_h) = 0, \forall q_h \in Q_h^0.
\]
Observing that $\gamma_p^{-1} h^{-2} \nabla \cdot w_h \in Q_h^0$ we
may eliminate the physical pressure from the formulation, since
\begin{multline*}
b(p_h,w_h) = -(p_h,\nabla\cdot w_h)_h = -s_{p,2}(p_h,\gamma_p^{-1} h^{-2}
\nabla \cdot w_h) =  b(\gamma_p^{-1} h^{-2}
\nabla \cdot w_h,z_h) \\
=  (\gamma_p^{-1} h^{-2}
\nabla \cdot w_h, \nabla \cdot z_h)_h.
\end{multline*}
Similarly, for $\gamma_x > 0$ the dual pressure $x_h$ may be
eliminated. Starting from the mass conservation equation
\[
-b(y_h,u_h) - s_{x,0}(x_h,y_h) = 0
\]
we use that $y_h = \nabla \cdot v_h$ is a valid test function to deduce
\begin{equation*}
-b(x_h,v_h) = -(x_h,\nabla\cdot v_h)_h = s_{x,0}(x_h,\nabla \cdot v_h) = -b(
\nabla \cdot v_h,u_h) 
= (\nabla \cdot v_h, \nabla \cdot u_h)_h.
\end{equation*}
The resulting formulation is an equal order interpolation formulation
for the Stokes' system using the nonconforming Crouzeix-Raviart
element for both the forward and the dual system. Find $(u_h,z_h) \in
V_h \times W_h$ such that
\begin{align} \label{BD_method}
a_h(u_h,w_h) - (\gamma_p^{-1} h^{-2}
\nabla \cdot w_h, \nabla \cdot z_h)_h & = l(w_h) \\
a_h(z_h,v_h) + (\nabla \cdot u_h, \nabla \cdot v_h)_h +
  s_{j,-1}(u_h,v_h) + m(u_h,v_h) & =
                                                                   m(\tilde
                                   u_M,v_h) \nonumber
\end{align}
for all  $(v_h,w_h) \in
V_h \times W_h$. We identify this scheme as a discretization of the continuous
regularization of the Stokes' Cauchy problem proposed in
\cite{BD14}. It follows that the analysis below also covers that
method in the special case that the discretization uses the
nonconforming space $X_h$.
\subsection{Technical Lemmas}
We will end this section by proving some elementary Lemmas that will
be useful in the analysis below. We will use $\|\cdot\|_X$ to denote
the $L^2$--norm over $X$, subset of $\mathbb{R}^d$ or $\mathbb{R}^{d-1}$.

We recall the interpolation operator $r_h: [H^1(\Omega)]^d
\rightarrow [X_h]^d$ defined by the (component wise) relation 
\[
\overline{\{r_h v\}}\vert_F := |F|^{-1} \int_F \{r_h v\}~\mbox{d}s
=  |F|^{-1}  \int_F  v ~\mbox{d}s 
\]
for every $F \in \mathcal{F}$ and with $|F|$ denoting the $(d-1)$-measure
of $F$.
It is conventient to introduce the broken scalar
  product 
\[
(x,y)_h := \sum_{\kappa \in \mathcal{T}_h} \int_\kappa x y ~\mbox{d}x,
\]
with the associated norms 
\[
\|x\|_h^2 := (x,x)_h \mbox{
  and } \|x\|_{1,h}^2 := \|x\|_h^2+a_h(x,x).
\]
The following inverse and trace inequalities are well known
\begin{equation}\label{trace}
\begin{array}{rcl}
\|v\|_{\partial \kappa} \leq C_t (h^{-\frac12} \|v\|_{\kappa} +
h^{\frac12} \|\nabla v\|_{\kappa}), \quad\forall v \in H^1(\kappa), \\[3mm]
h_{\kappa}\|\nabla v_h\|_{\kappa} + h_\kappa^{\frac12} \|v_h\|_{\partial \kappa}
\leq C_i \|v_h\|_\kappa, \quad \forall v_h \in X_h.
\end{array}
\end{equation}
Using the inequalities of \eqref{trace} and standard approximation
results from \cite{CR73} it is straightforward to show the following
approximation results of the interpolant $r_h$
\begin{equation}\label{eq:approx}
\begin{array}{rcl}
\|u - r_h u\|_\Omega    +h \|\nabla (u - r_h u)\|_h &\leq &C h^{t} |u|_{H^t(\Omega)}\\[3mm]
 \|h^{-\frac12} ( u - r_h
u) \|_{\mathcal{F}} + \|h^{\frac12} \nabla (u - r_h u) \cdot n_F
\|_{\mathcal{F}} &\leq &C h^{t-1} |u|_{H^t(\Omega)}
\end{array}
\end{equation}
where $t \in \{1,2\}$. It will also be useful to bound the $L^2$-norm
of the interpolant $r_h$ by its values on the element faces. To this
end we prove a technical lemma.
\begin{lemma}\label{lem:invtrace}
For any function $v_h \in X_h$ there holds
\[
\|v_h\|_\Omega \leq c_{\mathcal{T}}\left( \sum_{F \in \mathcal{F}} h_F \|\overline{
  \{v_h\}}\|^2_F \right)^{\frac12}
\]
\end{lemma}
\begin{proof}
It follows by norm equivalence of discrete
spaces on the reference element and a scaling argument (under the
assumption of shape regularity) that for all
$\kappa \in \mathcal{T}_h$
\begin{equation}\label{eq:inv_trace}
 \|v_h\|^2_{\kappa} \leq C \sum_{\substack{F \in \partial \kappa }}
h_F \|\overline{
  v}_h\|_F^2.
\end{equation}
The claim follows by summing over the elements of $\mathcal{T}_h$ and
recalling that $\|\overline{v}_h\|_F^2 = \|\overline{\{v_h\}}\|_F^2 $.
\end{proof}

For the analysis below we also need a quasi-interpolation operator
that maps piecewise linear, nonconforming functions into the space of
piecewise linear conforming functions. 
Let $I_\text{cf} [Q_h]^d \cup V_h \mapsto V \cap
V_h$ denote the quasi interpolation of $u_h$ into $V_h \cap V$, \cite{HW96,ABC03,KP03} such that
\[
\| I_\text{cf} u_h - u_h \|_\Omega+ h \|\nabla(I_\text{cf} u_h - u_h)\|_{h} \lesssim \|h^{\frac12} [u_h]\|_{\mathcal{F}_i}
\]
and for $v_h \in X_h$
\begin{equation}\label{disc_i1}
\|I_\text{cf} \nabla v_h - \nabla v_h\|_h \lesssim \|h^{\frac12}[\nabla v_h]\|_{\mathcal{F}_i}. 
\end{equation}
Below, the global conservation properties of this operator will be
important and we therefore propose the following perturbed
variant that satisfies a global conservation property.
We define the modified interpolant by
\begin{equation}\label{confcons_ip}
\tilde u_h :=  I_\text{cf} u_h + d_h 
\end{equation}
where the perturbation $d_h \in V_h \cap V$ is the solution to the following
constrained problem, $\bar p \in \mathbb{R}$
\begin{equation}\label{d_eq}
\begin{array}{rcl}
(d_h,w_h)_h+(\nabla d_h,\nabla w_h)_h + (\bar p,\nabla \cdot w_h)_h &=& 0 \\
(\nabla \cdot d_h, \bar q)_h & = & (-\nabla \cdot I_\text{cf} u_h, \bar q)_h,
\end{array}
\end{equation}
for all $(w_h,\bar q) \in (V_h \cap V) \times
\mathbb{R}$. This implies that
\[
\int_{\partial \Omega} \tilde u_h \cdot n ~\mbox{d}s = \int_{\Omega}
\nabla \cdot \tilde u_h ~\mbox{d}x = |\Omega| \overline{\nabla \cdot (I_\text{cf}
  u_h+d_h)} =  0
\]
with $|\Omega|$ denoting the $d$-measure of $\Omega$ and
\[
\overline{\nabla \cdot I_\text{cf}
  u_h} := |\Omega|^{-1} \int_{\Omega} \nabla \cdot I_\text{cf}
  u_h ~\mbox{d}x.
\]

\begin{lemma}\label{lem:is_conf}
The problem \eqref{d_eq} is well-posed and the
solution satisfies
\[
\|d_h\|_{H^1(\Omega)} \leq \|\overline{\nabla \cdot I_\text{\rm cf}
  u_h}\|_\Omega \lesssim \|h^{-\frac12} [u_h]\|_{\mathcal{F}_i} +
\|\overline{\nabla \cdot u_h}\|_h
\]
\end{lemma}
\begin{proof}
Since the linear system corresponding to \eqref{d_eq} is square,
existence and uniqueness is a consequence of the stability estimate.
Take $w_h = d_h + \alpha \bar p {\bf x}$, with $\alpha>0$, $\bar q =
\bar p$ in \eqref{d_eq} and observe that for
$\alpha$ small enough, there exists $c(\alpha)>0$ such that
\begin{multline}
\|d_h\|^2_{H^1(\Omega)}  + c(\alpha) \|\bar p\|_\Omega^2 \lesssim  \|\overline{\nabla \cdot I_\text{cf}
u_h}\|_\Omega^2 \lesssim \|\overline{\nabla \cdot (I_\text{cf}
u_h - u_h)}\|_h^2+ \|\overline{\nabla \cdot 
 u_h}\|_h^2 \\
\lesssim \|h^{-\frac12} [u_h]\|_{\mathcal{F}_i}^2 + \|\overline{\nabla \cdot 
 u_h}\|_h^2.
\end{multline}
\end{proof}
An immediate consequence of Lemma \ref{lem:is_conf} is that $\tilde
u_h$ satisfies similar approximation estimates as $I_\text{cf}$, but with
improved global conservation. We collect these results, the proof of
which is an immediate consequence of the discussion above, in a corollary.
\begin{corollary}
The conforming approximation $\tilde u_h$ satisfies the discrete estimate,
\begin{equation}\label{cons_interp}
\|\tilde u_h - u_h\|_{1,h} \lesssim  \|h^{-\frac12}
[u_h]\|_{\mathcal{F}_i} + \|\nabla \cdot u_h\|_h
\end{equation}
and has the global conservation property
\[
\int_{\partial \Omega} \tilde u_h \cdot n ~\text{\rm d} s = 0.
\]
\end{corollary}

Using the regularity assumptions on the data in $l(w)$
it is straightforward to show that the formulation satisfies the
following weak consistency
\begin{lemma}\label{lem:weakcons}(Weak consistency)
Let $(u,p)$ be the solution of \eqref{stokes}, with $f \in L^2(\Omega)$, and let $(u_h,p_h) \in \mathcal{V}$ be
the solution of \eqref{FEM:compact}. Then,  for all $w_h \in W_h$, there holds,
\begin{multline}\label{galortho}
|a_h(u_h - u,w_h) +  b_h(p_h - p, w_h) |\\
\leq \inf_{(\nu_h,\eta_h) \in \mathcal{V}} \sum_{F \in \mathcal{F}}
\int_{F}
| n_F\cdot (\sigma(u,p) - \{\sigma(\nu_h,\eta_h)\}) || [w_h]|
~\mbox{\rm d}s.
\end{multline}
where $\sigma(u,p) := \nabla u - \mathcal{I} p$,
with $\mathcal{I}$ the identity matrix.
\end{lemma}
\begin{proof}
Multiplying \eqref{stokes} with $w_h \in W_h$ and integrating by parts we have
\begin{multline}
\int_\Omega f w_h ~\mbox{d}x =-\int_\Omega \nabla \cdot
(\nabla u - \mathcal{I} p)\,\cdot\, w_h ~\mbox{d}x
\\
= -\sum_{\kappa \in \mathcal{T}_h}
\sum_{F \in \partial \kappa} \int_F \sigma(u,p)\,\cdot\, n_{\kappa}\,\cdot\, w_h ~
\mbox{d}s + a_h(u,w_h) + b_h(p, w_h)
\end{multline}
or by rearranging terms
\[
a_h(u,w_h) + b_h(p, w_h) = l(w_h) + \sum_{\kappa \in \mathcal{T}_h}
\sum_{F \in \partial \kappa} \int_F \sigma(u,p)\,\cdot\, n_{\kappa}\,\cdot w_h ~
\mbox{d}s.
\]
Using \eqref{FEM:compact} we obtain
\[
a_h(u_h - u,w_h) + b_h(p_h - p, w_h) 
= -\sum_{\kappa \in \mathcal{T}_h}
\sum_{F \in \partial \kappa} \int_F \sigma(u,p)\, \cdot\, n_{\kappa}\cdot w_h ~
\mbox{d}s.
\]
Since every internal face appears twice with different orientation of
$n_{\kappa}$ we have for all $\nu_h \in V_h$,
\[
\sum_{F \in \partial \kappa} \int_F \sigma(u,p) \cdot n_\kappa \cdot  w_h ~
\mbox{d}s = \sum_{F \in \partial \kappa } \int_F n_\kappa\,\cdot\, (\sigma(u,p) - \{\sigma(\nu_h,\eta_h) \}) \cdot  w_h ~
\mbox{d}s. 
\]
We now observe that by replacing $w_h$ with the jump $[w_h]$ we may write the sum over the
 faces of the mesh, replacing $n_{\kappa}$ by $n_F$. The conclusion follows by
 taking absolute values on both sides and moving the absolute values
 under the integral sign resulting in the desired inequality.
\end{proof}
\begin{lemma}\label{weak_cons_full}
Let $U:=(u,p) \in V \times Q^0$ denote the solution to
\eqref{eq:min1}-\eqref{eq:min2} with $\delta u = 0$. Then there holds
\begin{multline*}
|\mathcal{A}[(U-U_h,Z_h),(X_h,Y_h)] - \mathcal{S}_h[(U_h,Z_h),(X_h,Y_h)] +
m(\tilde u_\mathcal{M} - u_h,v_h)| \\
\leq  \inf_{(\nu_h,\eta_h) \in \mathcal{V}}\sum_{F \in \mathcal{F}}
\int_{F}
| n_F\,\cdot\, (\sigma(u,p) -
\{\sigma(\nu_h,\eta_h)\}) || [w_h]| ~\mbox{d}s
\end{multline*}
for all $(X_h,Y_h):= ([w_,y_h],[v_h,q_h]) \in \mathcal{V}\times \mathcal{W}$.
\end{lemma}
\begin{proof}
Subtract \eqref{FEM:compact} from
\eqref{eq:min1}-\eqref{eq:min2} and apply Lemma \ref{lem:weakcons} to
the equation for the primal variable $U$.
\end{proof}
\begin{lemma}\label{lem:rhH1proj}
For any $v \in [H^1(\Omega)]^d$, $y \in L^2(\Omega)$ and for all $w_h
\in W_h, \, q_h \in Q_h$ there holds
\[
a_h(v - r_h v, w_h) = 0,\, b_h(q_h, v - r_h v) = 0 \mbox{ and }
 b_h(y - \pi_0 y,w_h) = 0
\]
\end{lemma}
\begin{proof}
By integration by parts we have, using the orthogonality property on
the faces of $r_h$,
\[
a_h(v - r_h v, w_h) = \sum_{\kappa \in \mathcal{T}_h} \sum_{F
  \in \partial \kappa} \int_F (v - r_h v) \cdot (n_{\kappa} \cdot \nabla w_h
) ~\mbox{d}s = 0,
\]
\[
b_h(q_h,v - r_h v) = \sum_{\kappa \in \mathcal{T}_h} \sum_{F
  \in \partial \kappa} \int_F (v - r_h v) \cdot n_{\kappa} q_h ~\mbox{d}s = 0,
\]
and by definition
\[
b_h(p - \pi_0 p, w_h) = ( p - \pi_0 p , \nabla \cdot w_h
)_h = 0.
\]
\end{proof}
\begin{lemma}\label{lem:stab_press}
Let $(u_h,p_h) \in \mathcal{V}$ then there holds
\[
\|h^{\frac12} n_F\,\cdot\, [\nabla u_h]\|_{\mathcal{F}_i} + \|h^{\frac12}
[p_h]\|_{\mathcal{F}_i} \lesssim \|h^{\frac12} n_F\,\cdot\,
[\nabla u_h - \mathcal{I} p_h]\|_{\mathcal{F}_i} + \|\nabla \cdot 
u_h\|_\Omega + \|h^{-\frac12} [u_h]\|_{\mathcal{F}_i}.
\]
\end{lemma}
\begin{proof}
Let $u_{i}$, $i=1,\hdots,d$ denote the components of
$u_h$ and define the tangential projection of the gradient matrix on the face
$F$ by
$T \nabla u_h := (I- n_F \otimes n_F) \nabla u_h$ where $\otimes$ denotes outer product. Considering one face $F \in \mathcal{F}_i$ we have
\[
\|h^{\frac12}n_F\,\cdot\,
[\nabla u_h - \mathcal{I} p_h]\|_{F}^2 = \|h^{\frac12} n_F\,\cdot\,
[\nabla u_h] \|_{F}^2 +  \|h^{\frac12}[p_h]\|_{F}^2 - 2 \int_F h_F  n_F\,\cdot\,
[\nabla u_h] \cdot (n_F\,\cdot\, [\mathcal{I} p_h]) \mbox{d}s.
\]
The integrand of the last term of the right hand side may be written
\[
 n_F\,\cdot\,
[\nabla u_h] \cdot (n_F\,\cdot\, [\mathcal{I} p_h]) = [p_h] \sum_{i=1}^d
\sum_{j=1}^d n_{F,i} n_{F,j} [\partial_{x_j} u_{i}].
\]
By applying the following identity
\[
\sum_{i=1}^d
\sum_{j=1}^d n_{F,i} n_{F,j} \partial_{x_j} u_i=  \nabla \cdot u_h -\text{\rm{tr(}}T \nabla u_h),
\]
where $\text{tr}(T \nabla u_h)$ denotes the trace of $T \nabla u_h$,
we may write
\[
[p_h] \left(\sum_{i=1}^d
\sum_{j=1}^d n_{F,i} n_{F,j} [\partial_{x_j} u_i]\right) = [p_h]\left([\nabla \cdot u_h]-[\text{\rm{tr(}}T \nabla u_h)]\right).
\]
Observe that since the tangential component of the gradient of
the conforming approximation $I_\text{cf} u_h$ does not jump we have
\[
[\text{\rm{tr(}}T \nabla u_h)] = [\text{\rm{tr(}}T (\nabla u_h  - \nabla I_\text{cf} u_h)].
\]
Collecting these identities we obtain
\begin{multline*}
\int_F h_F  n_F\,\cdot\,
[\nabla u_h] \cdot n_F\,\cdot\, [\mathcal{I} p_h] \mbox{d}s = \int_F h_F
[p_h]\Bigl(
 [\nabla \cdot u_h]-[\text{\rm{tr(}}T \nabla (u_h - I_\text{cf} u_h))]\Bigr)\mbox{d}s \\
\leq \|h^{\frac12} [p_h]\|_{F} C_i (\|\nabla (u_h - I_\text{cf}
u_h)\|_{\Delta_F} +  \|\nabla \cdot u_h\|_{\Delta_F}),
\end{multline*}
where $\Delta_F$ denotes the union of the elements that have $F$ as
common face. Consequently
\[
2 \int_F h_F  n_F\,\cdot\,
[\nabla u_h] \cdot n_F\,\cdot\, [\mathcal{I} p_h] \mbox{d}s \leq \frac12
\|h^{\frac12} [p_h]\|_{F}^2 + C \|h^{-\frac12} [u_h]\|_{\mathcal{F}_{\Delta_F}}^2 + C
\|\nabla \cdot u_h\|^2_{\Delta_F}.
\]
Summing over $F \in \mathcal{F}_i$ we see that
\[
\|h^{\frac12} n_F\,\cdot\, [\nabla u_h]\|^2_{\mathcal{F}_i} + \frac12 \|h^{\frac12}
[p_h]\|^2_{\mathcal{F}_i} \lesssim \|h^{\frac12} n_F\,\cdot\,
[\nabla u_h - \mathcal{I} p_h]\|^2_{\mathcal{F}_i} +  \|\nabla \cdot 
u_h\|^2_\Omega + C \|h^{-\frac12} [u_h]\|^2_{\mathcal{F}_i}
\]
which proves the claim.
\end{proof}
\begin{lemma}\label{disc_poinc}(Discrete Poincar\'e inequality)
For all $(u_h,p_h) \in V_h \times Q_h^0$ there holds
\[
\|h u_h\|_{1,h} \lesssim \|h^{\frac12} n_F\,\cdot\, [\nabla u_h]\|_{\mathcal{F}_i} +
\|h^{-\frac12} [u_h]\|_{\mathcal{F}_i} + \|u_h\|_{\omega}
\]
and
\[
\| h p_h\|_\Omega \lesssim \|h^{\frac12} [p_h]\|_{\mathcal{F}_i}.
\]
\end{lemma}
\begin{proof}
For the first inequality use the Poincar\'e inequality for nonconforming finite elements
and a triangle inequality
\[
\|h u_h\|_{1,h} \lesssim \|h (\nabla u_h - I_\text{cf} \nabla u_h)\|_h +
\|h  I_\text{cf} \nabla u_h\|_h.
\]
Then observe that for $I_\text{cf} \nabla u_h$ constant,
$\|u_h\|_{\omega}=0$ implies that $I_\text{cf} \nabla u_h = 0$ and
therefore \cite[Lemma B.63]{EG04}
\[
\|h  I_\text{cf} \nabla u_h\|_h\leq \| h \nabla ( I_\text{cf} \nabla u_h -
\nabla u_h)\|_h+ \|u_h\|_{\omega}.
\]
Using \eqref{disc_i1} componentwise twice we then have
\[
\|h u_h\|_{1,h} \lesssim \|h^{\frac12} [\nabla u_h]\|_{\mathcal{F}_i}+ \|u_h\|_{\omega}.
\]
Finally each component of $\nabla u_h$ is decomposed on the normal and
tangential component on each face $F$ and we observe that using an elementwise
trace inequality, 
\begin{multline*}
 \|h^{\frac12} (\mathcal{I} - n_F \otimes n_F) [\nabla
 u_h]\|_{\mathcal{F}_i} = \|h^{\frac12} (\mathcal{I} - n_F \otimes n_F) [\nabla
 (u_h -  I_\text{cf} u_h)]\|_{\mathcal{F}_i} \\
\lesssim \|\nabla
 (u_h - I_\text{cf} u_h)\|_h \lesssim \|h^{-\frac12} [u_h]\|_{\mathcal{F}_i}.
\end{multline*}
Similarly for the proof of the second inequality observe that since
(redefining $I_\text{cf}$ to act on a scalar variable, and once again
by \cite[Lemma B.63]{EG04})
$\|h I_\text{cf} p_h\|_\Omega \lesssim  \|h \nabla  I_\text{cf}
p_h \|_\Omega  + \int_\Omega h  I_\text{cf}
p_h  ~\mbox{d}x$ 
there holds
\[
\|h p_h\|_h \lesssim \|h (p_h - I_\text{cf} p_h)\|_h + \|h \nabla ( I_\text{cf}
p_h - p_h)\|_{h} +\int_\Omega h  (I_\text{cf}
p_h - p_h) ~\mbox{d}x.
\]
It then follows using an inverse inequality that
\[
\|h p_h\|_h \lesssim \|p_h - I_\text{cf} p_h\|_\Omega  \lesssim \|h^{\frac12}[p_h]\|_{\mathcal{F}_i}
\]
and the proof is complete.
\end{proof}
\section{Stability estimates}\label{theory}
We will now focus on the formulation \eqref{FEM:compact} with
$\gamma_p=\gamma_x=0$. An immediate consequence of this choice is that
any solution to the system must satisfy 
\begin{equation}\label{eq:mass_conserv}
\nabla \cdot u_h\vert_\kappa  = \nabla
\cdot z_h\vert_\kappa = 0, \; \forall \kappa \in \mathcal{T}_h.
\end{equation}
The issue of stability of the discrete formulation is crucial since we
have no coercivity or inf--sup stability of the continuous formulation \eqref{eq:min1}--\eqref{eq:min2} to rely on. 
Indeed here the regularization plays an important part, since it
defines a semi-norm on the discrete space. We introduce a
mesh-dependent norm for the primal variable $X_h:= (v_h,q_h) \in \mathcal{V}$
\begin{equation}\label{def:tnorm1}
\tn X_h \tn_{V,Q} := \|h^{\frac12} n_F\cdot [\nabla v_h]\|_{\mathcal{F}_i}+ \|h^{\frac12} [q_h]\|_{\mathcal{F}_i} + \gamma_M^{\frac12}\|v_h\|_{\omega} + \|h^{-\frac12} [v_h]\|_{\mathcal{F}_i},
\end{equation}

We will also use the following triple norm with control of both the
dual pressure variabel $x_h$ and the dual velocities $z_h$.
\[
\tn (U_h,Z_h) \tn := \tn U_h \tn_{V,Q} +  \|x_h\|_\Omega + \|\nabla z_h \|_h .
\]
Since Dirichlet boundary conditions are set weakly on $W_h$, $\tn
 (U_h,Z_h) \tn$ can be shown to be a norm on $V_h\times Q^0_h$ using
 Lemmata \ref{lem:stab_press}--\ref{disc_poinc}. We
now prove a fundamental stability result for the discretization
\eqref{FEM:compact}.
\begin{theorem}\label{Thm:infsup}
  Let $\gamma_u, \gamma_M>0$, $\gamma_p=\gamma_x=0$ in \eqref{FEM:compact}--\eqref{Sglobal}. There exists a positive constant $c_s$, that is independent of $h$,
but not of $\gamma_u$, $\gamma_M$ or the local mesh geometry, such
that for all $(U_h,Z_h) \in \mathcal{V} \times \mathcal{W}$ there
holds
\[
c_s \tn (U_h,Z_h) \tn \leq \sup_{(X_h,Y_h) \in \mathcal{V} \times
  \mathcal{W}} \frac{\mathcal{G}[(U_h,Z_h),(X_h,Y_h) ]}{\tn (X_h,Y_h) \tn}
\]
\end{theorem}
\begin{proof}
First we observe that by testing with $X_h=U_h$ and $Y_h=-Z_h$ we have
\[
\gamma_u \|h^{-\frac12}[u_h]\|^2_{\mathcal{F}_i}+\gamma_M \|u_h\|^2_\omega = \mathcal{G}[(U_h,Z_h),(U_h,-Z_h) ].
\]
Then observe that by integrating by parts in the bilinear form
$a_h(\cdot,\cdot)$ and using the zero mean value property of the
approximation space we have
\[
a_h(u_h,w_h) + b_h(p_h,w_h) = \sum_{F \in \mathcal{F}_i} \int_F [n_F
\cdot \nabla u_h -  p_h n_F]
\cdot\{w_h\} ~\mbox{d}s.
\]
Define the function $\xi_h \in W_h$ such that for every face $F \in
\mathcal{F}_i$
$$
\overline{ \{\xi_h\}}\vert_{F} := h_F  [n_F \cdot \nabla u_h -  p_h n_F]\vert_F.
$$
This is possible in the
nonconforming finite element space since the degrees of freedom may be
identified with the average value of the finite element function on an
element face. 
Using Lemma \ref{lem:invtrace} we have
\begin{equation}\label{eq:L2_stab_func}
 \|\xi_h\|^2_{\Omega} \leq c_{\mathcal{T}}   \sum_{F \in \mathcal{F}_{i}} h^2_F  \|h_F^{\frac12}  [ n_F \cdot \nabla u_h -  p_h n_F]\|_F^2.
\end{equation}
Testing with $Y_h=(\xi_h,0)$ and $X_h = (0,0)$ we get
\begin{equation*}
\|h^{\frac12}[n_F \cdot\nabla u_h -  p_h n_F]\|_{\mathcal{F}_i}^2 
=\mathcal{G}[(U_h,Z_h),(0,Y_h) ].
\end{equation*}
By testing with $X_h = (z_h+\alpha r_h v_x,x_h)$, where $\alpha>0$ and
$v_x \in
[H^1(\Omega)]^d$ is a function such that $\nabla \cdot v_x = x_h$ and
$\|v_x\|_{H^1(\Omega)} \leq c_x \|x_h\|_{\Omega}$, we have
\begin{multline*}
\|\nabla z_h\|_h^2+ \alpha \|x_h\|^2+a_h(z_h,\alpha r_h v_x) \\+
\gamma_u s_{j,-1}(u_h,z_h+\alpha r_h v_x) + m(u_h,z_h+\alpha r_h
v_x)_\omega \\
= \mathcal{G}[(U_h,Z_h),(X_h,0) ].
\end{multline*}
Observe now that by the Cauchy-Schwarz inequality, the
arithmetic-geometric inequality and the stability of
$r_h v_x$ there holds
\[
a_h(z_h,\alpha r_h v_x) \leq \frac14 \|\nabla z_h\|_h^2 + c_x^2
\alpha^2 \|x_h\|_\Omega^2.
\]
Then by the trace
inequality and Poincar\'e's inequality
\begin{multline*}
\gamma_u s_{j,-1}(u_h,z_h+\alpha r_h v_x) + \gamma_M (u_h,z_h+\alpha r_h v_x)_\omega  \\
\lesssim (\gamma_u
\|h^{-\frac12} [u_h]\|_{\mathcal{F}_i} + \gamma_M \|u_h\|_\omega) (
\|\nabla z_h\|_h + \|v_x\|_{H^1(\Omega)})  \\
\leq C_\gamma (\gamma_u
\|h^{-\frac12} [u_h]\|_{\mathcal{F}_i}^2+  \gamma_M
\|u_h\|_\omega^2) + \frac{1}{4} (\|\nabla z_h\|^2_h + \alpha^2 \|x_h\|_{\Omega}^2).
\end{multline*}
The consequence of this is that for $\alpha,\beta>0$ sufficiently small there
exists $c$ such that
\begin{multline}\label{first_stab}
c \Bigl(\|h^{-\frac12} [u_h]\|_{\mathcal{F}_i}^2 +\|\nabla z_h\|_h^2+ \|u_h\|_{\omega}^2 \\+
  \| x_h\|^2_{\Omega}+
\|h^{\frac12}[n_F \cdot \nabla x_h - p_h n_F ]\|_{\mathcal{F}}^2 \Bigr)
\\
\leq  \mathcal{G}[(U_h,Z_h),(X_{UZ},Y_{UZ})],
\end{multline}
where $X_{UZ} = U_h+(\beta (z_h+ \alpha r_h v_x),x_h)$, $Y_{UZ} =-Z_h + (\xi_h,0)$.
Applying Lemma \ref{lem:stab_press}, recalling that $\|\nabla \cdot
u_h\|_h = 0$ we deduce that
\begin{equation}
C \tn (U_h,Z_h) \tn^2 \leq \mathcal{G}[(U_h,Z_h),(X_{UZ},Y_{UZ})].
\end{equation}
It remains to prove that $\tn(X_{UZ},Y_{UZ}) \tn \lesssim \tn(U_h,Z_h) \tn$.
This follows by observing that
\begin{multline*}
\tn(X_{UZ},Y_{UZ}) \tn \leq \tn (U_h,Z_h)\tn \\
+\beta ( \|h^{\frac12}
[\nabla(z_h+ \alpha r_h v_x)]\|_{\mathcal{F}_i}+ \|(z_h+ \alpha r_h v_x)\|_{\omega}
+ \|h^{-\frac12}[z_h+ \alpha r_h v_x]\|_{\mathcal{F}_i}) + \|\nabla \xi_h\|_h
\end{multline*}
and
\begin{multline*}
 \|h^{\frac12}
[\nabla(z_h+ \alpha r_h v_x)]\|_{\mathcal{F}_i}+ \|(z_h+ \alpha r_h v_x)\|_{\omega}
+ \|h^{-\frac12}[z_h+ \alpha r_h v_x]\|_{\mathcal{F}_i} \lesssim \| z_h\|_{1,h}+\| r_h v_x\|_{1,h} \\
\lesssim \|\nabla z_h\|_h + \|x_h\|_\Omega.
\end{multline*}
Finally we use an inverse inequality and the bound
\eqref{eq:L2_stab_func} to obtain the bound
\[
 \|\nabla \xi_h\|_h \lesssim  \|h_F^{\frac12}  [ n_F \cdot \nabla u_h
 -  p_h n_F]\|_{\mathcal{F}_i} \lesssim \tn (U_h,0)\tn
\]
which finishes the proof.
\end{proof}
\begin{corollary}
The formulation \eqref{FEM:compact} admits a unique solution $(u_h,p_h)
\in \mathcal{V}$ and $(z_h,x_h) \in \mathcal{W}$.
\end{corollary}
\begin{proof}
The system matrix corresponding to \eqref{FEM:compact} is a square matrix
and we only need to show that there are no zero eigenvalues.
Assume that $l(w_h) = 0$. It then follows by Theorem \ref{Thm:infsup}
that for any solution $(u_h,p_h)$ there holds
\[
c_s \tn (U_h,Z_h) \tn \leq \sup_{(X_h,Y_h) \in \mathcal{V} \times
  \mathcal{W}} \frac{\mathcal{G}[(U_h,Z_h),(X_h,Y_h) ]}{\tn (X_h,Y_h)
  \tn} =0.
\]
Recalling Lemma \ref{disc_poinc} this
implies that $u_h=p_h=z_h=x_h=0$ showing that the solution is unique.
\end{proof}
\begin{remark}
Observe that the proof of Theorem \ref{Thm:infsup} works also for $\gamma_p
> 0$ and $\gamma_x > 0$, the only modification in this case is
that the contribution $\|\nabla \cdot u_h\|_h$ must be added to the
norm $\tn(u_h,p_h)\tn_V$ and stability must be proven by testing
with $y_h = \nabla \cdot u_h$.
\end{remark}
\section{Error estimates using conditional stability}\label{error_est}
In this section we will use the stability proven in the previous
section to derive error estimates. 
\begin{proposition}\label{prop:stab_conv}
Let $(u,p) \in [H^2(\Omega)]^d \times H^1(\Omega)$ be the solution of
\eqref{stokes} and $(u_h,p_h) \times Z_h$ the solution to \eqref{FEM:compact}--\eqref{Sglobal},
with $\gamma_u,\gamma_M>0$ and $\gamma_p = \gamma_x = 0$. Then there holds
\[
\tn ((r_h u - u_h, \pi_0 p - p_h),Z_h)\tn \lesssim  h (\|u\|_{H^2(\Omega)} +
\|p\|_{H^1(\Omega)})+ \gamma_M^{\frac12} \|\delta u\|_{\omega}
\]
\end{proposition}
\begin{proof}
First denote the discrete
error $\Theta_h =(r_h u - u_h,\pi_0 p - p_h)$. Then by Theorem \ref{Thm:infsup}
\[
c_s \tn(\Theta_h,Z_h)\tn^2 \leq  
\sup_{(X_h,Y_h) \in \mathcal{V}\times \mathcal{W}} \frac{\mathcal{G}[(\Theta_h,Z_h),(X_h,Y_h)]}{\tn (X_h,Y_h) \tn}.
\] 
Then applying Lemma \ref{weak_cons_full} and \ref{lem:rhH1proj} we have
\begin{multline*}
\mathcal{G}[(\Theta_h,Z_h),(X_h,Y_h) ] \leq \inf_{(\nu_h,\eta_h) \in V_h \times W_h}\sum_{F \in \mathcal{F}}
\int_{F}
|(\sigma(u,p) - \{\sigma(\nu_h,\eta_h)\}) \cdot n_F || [w_h]|
~\mbox{d}s \\
-b_h(y_h,r_h u - u)  +s_{j,-1}(r_h u,v_h)
+ \gamma_M(r_h u - u - \delta u, v_h)_{\omega}.
\end{multline*}
First note that
\begin{multline*}
 \inf_{(\nu_h,\eta_h) \in \mathcal{V}}\sum_{F \in \mathcal{F}}
\int_{F}
|(\sigma(u,p) - \{\sigma(\nu_h,\eta_h)\}) \cdot n_F || [w_h]|
~\mbox{d}s  \\
\leq h^{\frac12} (\inf_{(\nu_h,\eta_h) \in \mathcal{V}}  \sum_{F \in \mathcal{F}} \|\sigma(u,p) -
  \{\sigma(\nu_h,\eta_h)\}\|_{F}^2)^{\frac12} \|\nabla w_h\|_h \\
\lesssim h (\|u\|_{H^2(\Omega)} + \|p\|_{H^1(\Omega)}) \tn (0, Y_h)\tn,
\end{multline*}
\[
b_h(y_h,r_h u - u) = 0,
\]
\[
s_{j,-1}(r_h u,v_h) \leq  C h \|u\|_{H^2(\Omega)} \|h^{-\frac12} [v_h]\|_{\mathcal{F}_i} \leq C h
\|u\|_{H^2(\Omega)} \tn(X_h,0)\tn.
\]
Finally, using a Cauchy-Schwarz inequality and a Poincar\'e
inequality for $\eta_h$
\[
\gamma_M(r_h u - u, v_h)_{\omega} \lesssim \gamma_M\|r_h u
- u\|_{\omega} \|v_h\|_{\omega}\lesssim
h^2 \|u\|_{H^2(\Omega)} \tn(X_h,0)\tn.
\]
For the perturbation we have
\[
\gamma_M (\delta u, w_h)_{\omega}  \leq \gamma_M \|\delta u\|_{\omega} \|w_h\|_{\omega}.
\]
Collecting the above estimates ends the proof.
\end{proof}
\begin{theorem}\label{thm:asymptotic}
Assume that $u\in [H^2(\Omega)]^d$ and $p \in H^1(\Omega)$. Let
$\tilde u_h$ be defined by \eqref{confcons_ip} then
\[
\|\tilde u_h\|_{H^1(\Omega)} + \|p_h\|_\Omega \lesssim
\|u\|_{H^2(\Omega)}+\|p\|_{H^1(\Omega)} + \gamma_M^{\frac12} h^{-1} \|\delta u\|_\omega
\]
and, if $\delta u = 0$
\[
\tilde u_h \rightharpoonup u \mbox{ in } [H^1(\Omega)]^d \mbox{ and }
p_h \rightharpoonup p \mbox{ in } L^2\Omega).
\]
\end{theorem}
\begin{proof}
For the pressure we immediately observe that
\[
\|p_h \|_\Omega \lesssim \|p_h - \pi_0 p\|_\Omega +
\|p\|_\Omega 
\lesssim h^{-1} \tn  (0,p_h - \pi_0 p)\tn_V +
\|p\|_\Omega
\]
Then observe that by a Poincar\'e inequality and the $H^1$-stability
of the interpolation operator $r_h$ there holds
\begin{multline*}
\|\tilde u_h\|_{H^1(\Omega)} \leq \|\tilde u_h - u_h\|_{1,h} + \|
u_h\|_{1,h}\\
\leq \|\tilde u_h - u_h\|_{1,h} + \|
u_h - r_h u\|_{1,h}+ \|r_h
u\|_{1,h} \\
\lesssim \|h^{-\frac12}[u_h - r_h u]\|_{\mathcal{F}_i}+ \|u_h - r_h u\|_{\omega} + \|\nabla
(u_h - r_h u)\|_h + \|u\|_{H^1(\Omega)}.
\end{multline*}
Therefore
\[
\|\tilde u_h\|_{H^1(\Omega)} \lesssim h^{-1} \tn (u_h - r_h u,0)\tn_V +
\|u\|_{H^1(\Omega)}
\]
and the first claim follows by applying Proposition
\ref{prop:stab_conv}.

It follows that for $\delta u=0$ we may extract a subsequence of pairs $(\tilde
u_h,p_h)$ that converges weakly in $[H^1(\Omega)]^d \times L^2(\Omega)$.
By construction the divergence of the $H^1$-conforming part satisfies
\[
\|\nabla \cdot \tilde u_h\|_\Omega \lesssim \|h^{-\frac12}[u_h - r_h
u]\|_{\mathcal{F}_i}+ \underbrace{\|\nabla \cdot u_h\|_h}_{=0}+ h
\|u\|_{H^2(\Omega)}\lesssim h
\|u\|_{H^2(\Omega)}
\]
and hence $\|\nabla \cdot \tilde u\|_h \rightarrow 0$ for $h
\rightarrow 0$.
 It remains to show
that the weak limit is a weak solution of Stokes equation. To this end
consider, with $w \in C^1_0(\Omega)$,
\begin{multline*}
|a(\tilde u_h,w) + b(p_h, w) - l(w)|
\\ =
|a_h(\tilde u_h - u_h,w)+ a_h(u_h,w - r_h w)+ b(p_h, w - r_h w) - l(w - r_h w)| \\
= |a_h(\tilde u_h - u_h,w) - l(w - r_h w)| \lesssim
(\|h^{-\frac12}[u_h]\|_{\mathcal{F}_i} + h \|f\|_\Omega )
\|w\|_{H^1(\Omega)}\\
\lesssim h \|w\|_{H^1(\Omega)}
\end{multline*}
We conclude by taking the limit $h \rightarrow 0$.

\end{proof}
\begin{theorem}\label{thm:error_est}
Assume that $(u,p)\in [H^2(\Omega)]^d\times H^1(\Omega)$ is the unique
solution of \eqref{stokes} with $u=u_M$ in $\omega$ and the parameters $R_1, R_2$ and $R_3$ satisfy the assumptions of
Theorem \ref{thm:3sphere}. If $u_h$ is the solution of
\eqref{FEM:compact}-\eqref{Sglobal}, with $\gamma_u, \gamma_M >0$, $\gamma_p=\gamma_x=0$ and $\|\delta
u\|_\Omega \leq h_0$, $h_0>0$,
then, for $h>h_0$, there holds
\[
\|u - u_h\|_{B_{R_2}(x_0)} \lesssim  h^{\tau}
\]
where $\tau$ is the power from Theorem \ref{thm:3sphere} and the
hidden constant depends on $R_2/R_3$, the local mesh geometry and
$\|u\|_{H^2(\Omega)}$ and $\|p\|_{H^1(\Omega)}$.
\end{theorem}
\begin{proof}
First let $u - u_h = \underbrace{u - \tilde u_h}_{e_u \in [H^1(\Omega)]^d} +
\underbrace{\tilde u_h - u_h}_{e_h \in V_h}$, where $\tilde u_h$ is
defined by \eqref{confcons_ip}. We recall that
$$\|e_h\|_\Omega \lesssim \|h^{-\frac12}[u_h]\|_{\mathcal{F}_i} \leq C h
\|u\|_{H^2(\Omega)}$$ so we only need to bound $\|
e_u\|_{B_{R_2}(x_0)} $. Also introduce $e_p = p- p_h \in
L^2(\Omega)$. It follows that $(e_u,e_p)$ is a solution to the Stokes'
equation on weak form with a particular right hand side. Indeed we have for all $(w,q) \in [H^1_0(\Omega)]^d \times Q$
\begin{equation}\label{pert1}
a(e_u,w) + b(e_p, w) = l(w) - a(\tilde u_h,w)+b(p_h,
w) =:\left<\mathfrak{f},w \right>_{V',V}
\end{equation}
and 
\begin{equation}\label{pert2}
-b(q, e_u) = b (q, \tilde u_h) =: (\mathfrak{g},q)_{\Omega}
\end{equation}
where $\mathfrak{f} \in V'$ and $\mathfrak{g}\in L^2_0(\Omega)$. Now
consider the problem \eqref{stokes} with homogeneous Dirichlet
boundary conditions on $\partial \Omega$ and the right hand side $\mathfrak{f}$ and
$\mathfrak{g}$ as defined above. This problem is well-posed and we
call its solution $\{\mathcal{E}_u, \mathcal{E}_p\} \in
[H^1_0(\Omega)]^d\times L^2_0(\Omega)$. By the well-posedness of the
problem we know that
\[
\|\mathcal{E}_u\|_{H^1(\Omega)} + \|\mathcal{E}_p\|_\Omega \leq
\|\mathfrak{f}\|_{H^{-1}(\Omega)} + \|\mathfrak{g}\|_{\Omega} 
\]
We know from equation \eqref{cons_interp}, the fact that $\|\nabla
\cdot u_h\|_h=0$ and Proposition \ref{prop:stab_conv} that $\|\mathfrak{g}\|_{\Omega}
\lesssim \|h^{-\frac12} [u_h]\|_{\mathcal{F}_i} \lesssim h$ and for $\|\mathfrak{f}\|_{V'}$ we derive the bound
\begin{multline}\label{conv_e_u}
\sup_{\substack{w \in [H^1_0]^d \\ \|w\|_1 = 1}} \left<\mathfrak{f},w
\right>_{V'V} =  l(w) -   a(\tilde u_h,w)-b(p_h,
w) \\
= l(w - r_h w)-a_h(\tilde u_h - u_h,w)   \\ \lesssim
h \|f\|_\Omega + s_{j,-1}(u_h,u_h)^{\frac12} \lesssim h + \|\delta u\|_\omega.
\end{multline}
Considering now the functions $U:= u - \tilde u_h - \mathcal{E}_u$ and
$P:= p - p_h - \mathcal{E}_p$ we see that $\{U,P \}$ is a solution to
\eqref{stokes} with $\mathfrak{f}=0$ and $\mathfrak{g}=0$. By
equation \eqref{eq:elliptic_reg} we have $\{U,P\} \in [H^2(\varpi)]^d\times
H^1(\varpi)$ on every compact $\varpi \subset \Omega$. We
may then apply Theorem \ref{thm:3sphere} to $U$ and obtain
\begin{equation}\label{3sphere_app}
\int_{B _{R_2}(x_0)} |U|^2 ~\mbox{d}x \leq C \left( \int_{B _{R_1}(x_0)} |U|^2 ~\mbox{d}x\right)^{\tau} \left(\int_{B_{R_3}(x_0)} |U|^2 ~\mbox{d}x \right)^{1-\tau}.
\end{equation}
These results may now be combined in the following way to prove the
theorem. First by the triangle inequality, writing $u-u_h = U +
\mathcal{E}_u + \tilde u_h - u_h$,
\[
\|u - u_h\|_{B_{R_2}(x_0)} \leq 
\|\mathcal{E}_u\|_{B_{R_2}(x_0)}+ \|\tilde u_h - u_h\|_{B_{R_2}(x_0)}+\|U\|_{B_{R_2}(x_0)}
= I +II +III.
\]
By \eqref{Cont_dep_well_posed} and \eqref{conv_e_u} there holds for the first term
\[
I \lesssim \|\mathcal{E}_u\|_{H^1(\Omega)}
  \lesssim   h + \|\delta u\|_\omega
\]
and using the discrete interpolation and Proposition \ref{prop:stab_conv}
\[
II = \|\tilde u_h - u_h\|_{B_{R_2}(x_0)} \lesssim \|h^{-\frac12} [u_h]\|_{\mathcal{F}_i} \lesssim  h + \|\delta u\|_\omega.
\]
For the last term, using \eqref{3sphere_app},
we have
\[
III \lesssim \left( \int_{B _{R_1}(x_0)} |U|^2 ~\mbox{d}x\right)^{\tau/2} \left(\int_{B_{R_3}(x_0)} |U|^2 ~\mbox{d}x \right)^{(1-\tau)/2}.
\]
By the definition of $U$ and since by assumption $B _{R_1}(x_0)
\subset \omega$
\begin{multline}
\left(\int_{B _{R_1}(x_0)} |U|^2 ~\mbox{d}x\right)^{\frac12} \lesssim \|r_h u - u_h\|_\omega
+ \|r_hu - u\|_\omega + \|\tilde u_h - u_h\|_\omega \\
+ \|\mathcal{E}_u\|_{B_{R_1}(x_0)} \lesssim h + \|\delta u\|_\omega.
\end{multline}
Here we applied Proposition \ref{prop:stab_conv},
  \eqref{eq:approx}, discrete interpolation \eqref{cons_interp}, and
  \eqref{Cont_dep_well_posed} applied to $\mathcal{E}_u$.
Finally by the triangle inequality, the a priori assumption $u \in
H^2(\Omega)$,  \eqref{Cont_dep_well_posed}  and the first claim of
Theorem \ref{thm:asymptotic} we have
\[
\left(\int_{B_{R_3}(x_0)} |U|^2 ~\mbox{d}x \right)^{\frac12} \leq
\|u\|_{H^1(\Omega)} + \|\tilde
u_h\|_{H^1(\Omega)}+\|\mathcal{E}_u\|_{H^1(\Omega)} \lesssim 1 +
h^{-1} \|\delta u\|_\omega.
\]
The claim follows by collecting the bounds on the terms $I-III$ and
applying the assumption on the perturbations in data versus the
mesh-size.
\end{proof}
\begin{remark}
It is straightforward to prove the Proposition \ref{prop:stab_conv}
and the Theorems \ref{thm:asymptotic} and \ref{thm:error_est} also for
$\gamma_p \ge 0$ and $\gamma_x \ge 0$ and thereby extending the
analysis to include the method \eqref{BD_method}. We leave the details for the reader.
\end{remark}
\begin{remark}
One may also introduce perturbations in the right hand side
$f$. Provided these perturbations are in $[L^2(\Omega)]^d$ the same
results holds. Details on the necessary modifications can be found in \cite{Bu15}.
\end{remark}
\section{Numerical example}
Our numerical example is set in the unit square
$\Omega = (0,1)^2$ with zero right hand side and data given in the
disc $S_{1/2}:=\{(x,y) \in\mathbb{R}^2: \sqrt{(x-0.5)^2 +
  (y-0.5)^2} < 0.125\}.$
The flow is nonsymmetric with the exact solution given by
$$
u(x,y) = (20x y^3, 5x^4 - 5 y^4) \qquad\text{and}\qquad 
 p(x,y) = 60 x^2 y - 20 y^3 -5.
$$

We consider the formulation \eqref{FEM:compact}-\eqref{Sglobal},
with $l(w_h)=0$ For the parameters we chose, $\gamma_M=800$ and
$\gamma_u=10^{-5}$, $\gamma_p=\gamma_z=\gamma_x=0$. First we perform
the computation with unperturbed data. The results are presented in
the left graphic of Figure \ref{fig:exa2}. We report the velocity error both in the global
$L^2$-norm (open square markers),
the local $L^2$-norm in the subdomain where $\sqrt{(x-0.5)^2 +
  (y-0.5)^2} < 0.375$ (filled square markers) and in the residual
quantities of \eqref{residual}
(circle markers, $r_1$ filled, $r_2$ open),
 \begin{equation}\label{residual}
r_1:=\left(\int_{S_{1/2}} (u_h - u)^2 ~\mbox{d}x\right)^{\frac12} \mbox{
and }
r_2:=\|h^{-\frac12}[u_h]\|_{\mathcal{F}_i}.
\end{equation}
The global pressure is plotted with triangle markers.
%
The error plots for this case are given in figure \ref{fig:exa2}.
We observe the $O(h)$ convergence of the residual
quantities \eqref{residual}. The global velocity and pressure $L^2$-errors appears to have
approximately $O(|log(h)|^{-1})$ convergence. The local error
matches the result of Theorem \ref{thm:error_est}. Indeed the dotted line is shows
the behavior of the quantity $C_1 \|e\|_{\Omega}^{0.3} (r_1+r_2)^{0.7} +
10 h^2$ illustrating the different components of the local error used in the proof of the theorem. We see that this quantity (with a properly chosen constant)
gives a good fit with the local error. 

The same computation was repeated with a $1\%$ relative random
perturbation of data. The results for this case is reported in the
right plot of Figure \ref{fig:exa2}.
As predicted by theory the results are stable under perturbation of data as long as the discretization error is larger than
the random perturbation (up to a constant). When the perturbations
dominate the errors in all quantities appear to stagnate.
\begin{figure}
\begin{center}
\includegraphics[height=6cm]{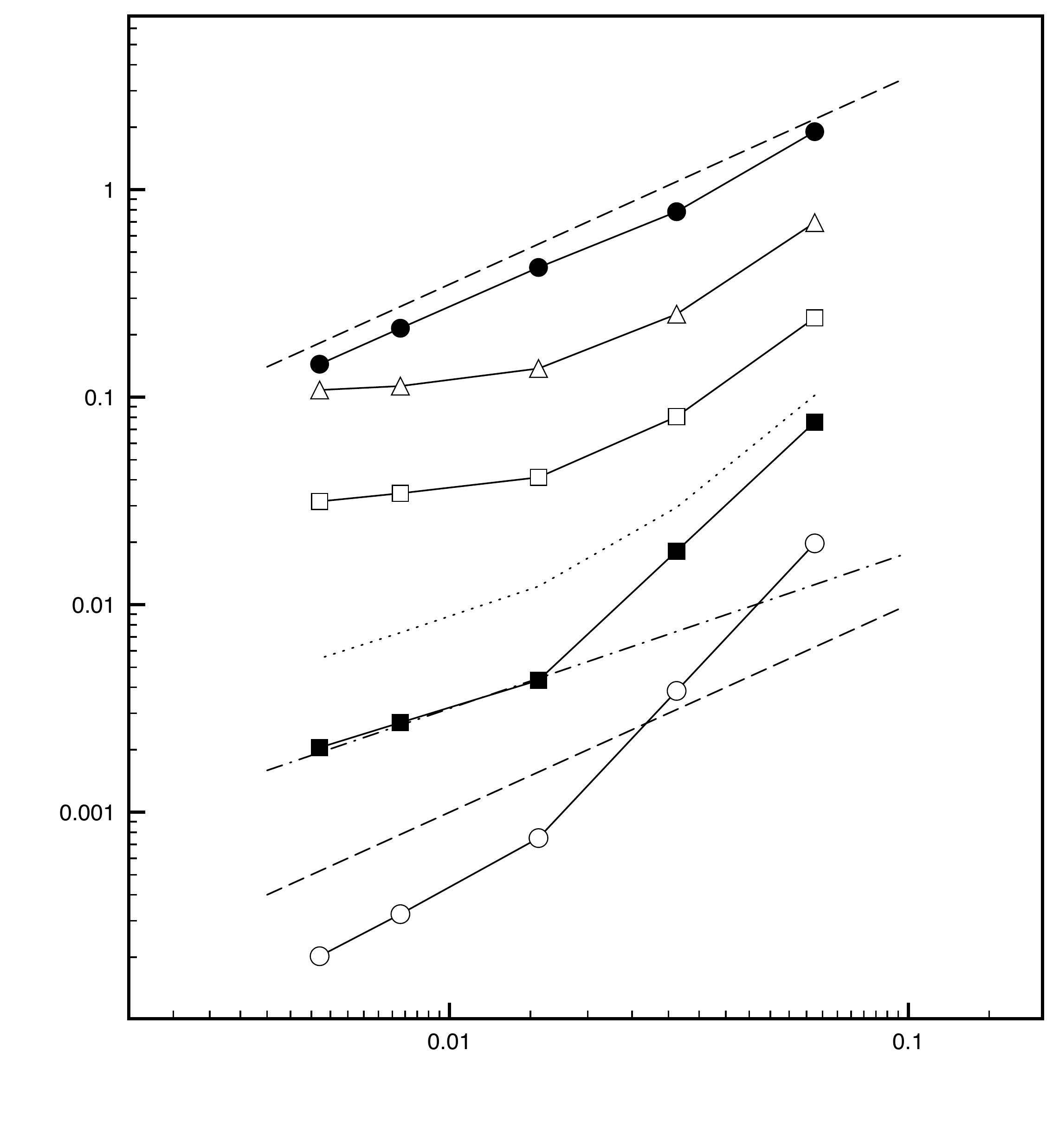}\includegraphics[height=6cm]{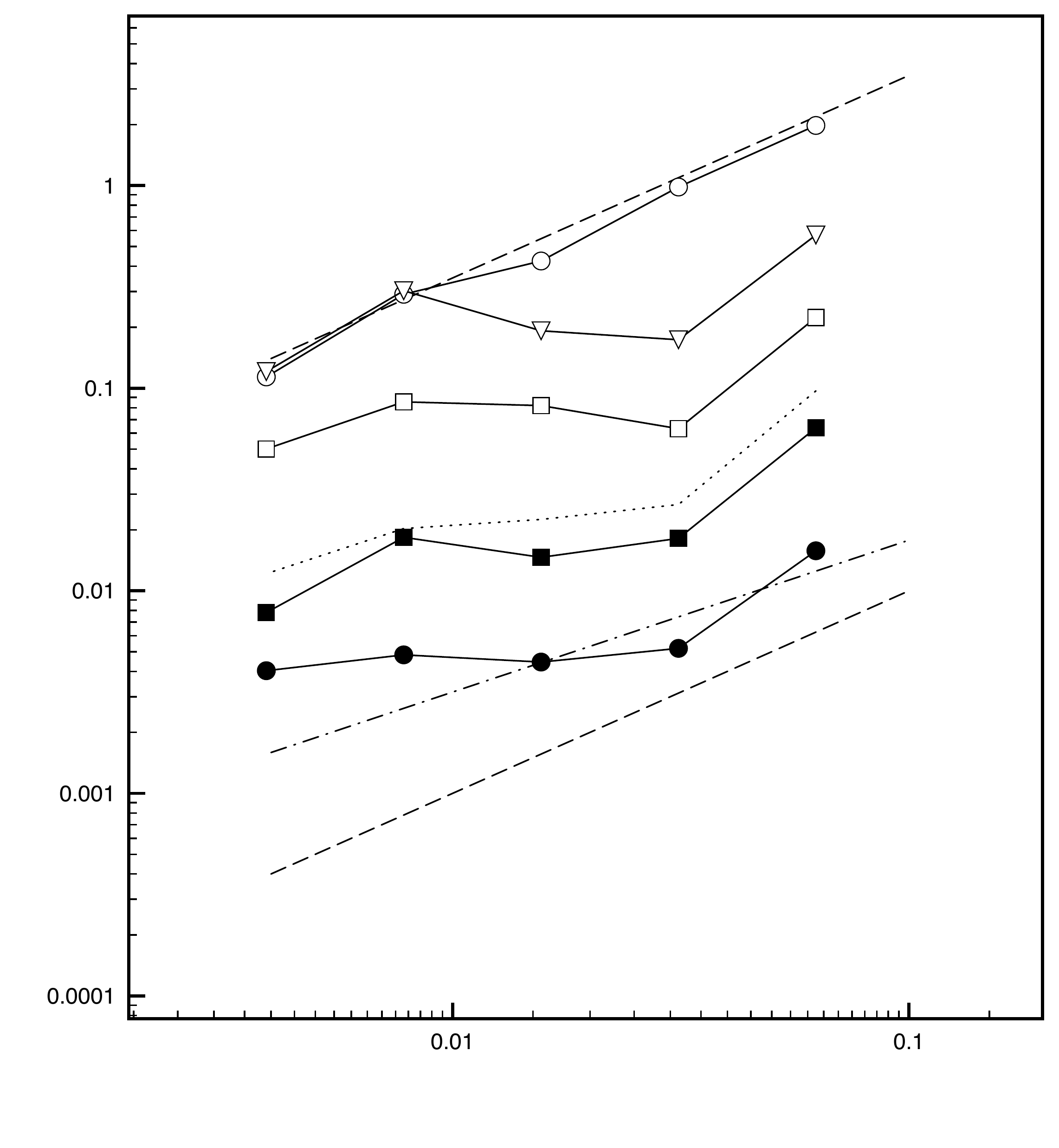}
\end{center}
\caption{Relative $L^2$-error  against mesh-size, left unperturbed
  data, right with $1\%$ relative noise. Reference lines are the same in both
  plots and of orders, dashed lines $\approx
  O(h)$ with different constants, dash dot $\approx
  O(h^{\frac34})$
  and dotted $C_1 \|e\|_{\Omega}^{0.3} (r_1+r_2)^{0.7} +
10 h^2$}\label{fig:exa2}
\end{figure}


\newpage
\begin{thebibliography}{10}

\bibitem{ASH09}
B.~A. {Abda}, I.~B. {Saad}, and M.~{Hassine}.
\newblock {Data completion for the Stokes system}.
\newblock {\em Comptes Rendus Mecanique}, 337(9--10):703--708, 2009.

\bibitem{ABC03}
Y.~Achdou, C.~Bernardi, and F.~Coquel.
\newblock A priori and a posteriori analysis of finite volume discretizations
  of {D}arcy's equations.
\newblock {\em Numer. Math.}, 96(1):17--42, 2003.

\bibitem{ARRV09}
G.~Alessandrini, L.~Rondi, E.~Rosset, and S.~Vessella.
\newblock The stability for the {C}auchy problem for elliptic equations.
\newblock {\em Inverse Problems}, 25(12):123004, 47, 2009.

\bibitem{BBF13}
D.~Boffi, F.~Brezzi, and M.~Fortin.
\newblock {\em Mixed finite element methods and applications}, volume~44 of
  {\em Springer Series in Computational Mathematics}.
\newblock Springer, Heidelberg, 2013.

\bibitem{BEG13}
M.~Boulakia, A.-C. Egloffe, and C.~Grandmont.
\newblock Stability estimates for the unique continuation property of the
  {S}tokes system and for an inverse boundary coefficient problem.
\newblock {\em Inverse Problems}, 29(11):115001, 21, 2013.

\bibitem{BD14}
L.~Bourgeois and J.~Dard{\'e}.
\newblock The ``exterior approach'' to solve the inverse obstacle problem for
  the {S}tokes system.
\newblock {\em Inverse Probl. Imaging}, 8(1):23--51, 2014.

\bibitem{Bu15}
E.~{Burman}.
\newblock {A stabilized nonconforming finite element method for the elliptic
  Cauchy problem}.
\newblock {\em Math. Comp. (2017, to appear)}.

\bibitem{Bu13}
E.~Burman.
\newblock Stabilized finite element methods for nonsymmetric, noncoercive, and
  ill-posed problems. {P}art {I}: {E}lliptic equations.
\newblock {\em SIAM J. Sci. Comput.}, 35(6):A2752--A2780, 2013.

\bibitem{Bu14}
E.~Burman.
\newblock Error estimates for stabilized finite element methods applied to
  ill-posed problems.
\newblock {\em C. R. Math. Acad. Sci. Paris}, 352(7-8):655--659, 2014.

\bibitem{BH05}
E.~Burman and P.~Hansbo.
\newblock Stabilized {C}rouzeix-{R}aviart element for the {D}arcy-{S}tokes
  problem.
\newblock {\em Numer. Methods Partial Differential Equations}, 21(5):986--997,
  2005.

\bibitem{CR73}
M.~Crouzeix and P.-A. Raviart.
\newblock Conforming and nonconforming finite element methods for solving the
  stationary {S}tokes equations. {I}.
\newblock {\em Rev. Fran\c caise Automat. Informat. Recherche Op\'erationnelle
  S\'er. Rouge}, 7(R-3):33--75, 1973.

\bibitem{EG04}
A.~Ern and J.-L. Guermond.
\newblock {\em Theory and practice of finite elements}, volume 159 of {\em
  Applied Mathematical Sciences}.
\newblock Springer-Verlag, New York, 2004.

\bibitem{FL96}
C.~Fabre and G.~Lebeau.
\newblock Prolongement unique des solutions de l'equation de {S}tokes.
\newblock {\em Comm. Partial Differential Equations}, 21(3-4):573--596, 1996.

\bibitem{GR86}
V.~Girault and P.-A. Raviart.
\newblock {\em Finite element methods for {N}avier-{S}tokes equations},
  volume~5 of {\em Springer Series in Computational Mathematics}.
\newblock Springer-Verlag, Berlin, 1986.

\bibitem{HL02}
P.~Hansbo and M.~G. Larson.
\newblock Discontinuous {G}alerkin methods for incompressible and nearly
  incompressible elasticity by {N}itsche's method.
\newblock {\em Comput. Methods Appl. Mech. Engrg.}, 191(17-18):1895--1908,
  2002.

\bibitem{HL03}
P.~Hansbo and M.~G. Larson.
\newblock Discontinuous {G}alerkin and the {C}rouzeix-{R}aviart element:
  application to elasticity.
\newblock {\em ESAIM: Math. Model. Numer. Anal.}, 37(1):63--72, 2003.

\bibitem{HW96}
R.~H.~W. Hoppe and B.~Wohlmuth.
\newblock Element-oriented and edge-oriented local error estimators for
  nonconforming finite element methods.
\newblock {\em RAIRO Mod\'el. Math. Anal. Num\'er.}, 30(2):237--263, 1996.

\bibitem{KP03}
O.~A. Karakashian and F.~Pascal.
\newblock A posteriori error estimates for a discontinuous {G}alerkin
  approximation of second-order elliptic problems.
\newblock {\em SIAM J. Numer. Anal.}, 41(6):2374--2399, 2003.

\bibitem{LUV10}
C.-L. Lin, G.~Uhlmann, and J.-N. Wang.
\newblock Optimal three-ball inequalities and quantitative uniqueness for the
  {S}tokes system.
\newblock {\em Discrete Contin. Dyn. Syst.}, 28(3):1273--1290, 2010.

\bibitem{Ser15}
G.~Seregin.
\newblock {\em Lecture notes on regularity theory for the {N}avier-{S}tokes
  equations}.
\newblock World Scientific Publishing Co. Pte. Ltd., Hackensack, NJ, 2015.

\bibitem{WW15}
C.~{Wang} and J.~{Wang}.
\newblock {A Primal-Dual Weak Galerkin Finite Element Method for Second Order
  Elliptic Equations in Non-Divergence Form}.
\newblock {\em ArXiv e-prints}, October 2015.

\end{thebibliography}
\end{document}